\documentclass[12pt, reqno]{amsart}
\usepackage{amssymb,a4}
\usepackage{amsmath,amsthm,amscd}
\usepackage{amsthm}
\usepackage{amsfonts,amssymb}
\usepackage{mathrsfs}
\usepackage{amsmath}
\numberwithin{equation}{section}
\newtheorem{theo}{Theorem}[section]
\newtheorem{defi}[theo]{Definition}
\newtheorem{coro}[theo]{Corollary}
\newtheorem{lemm}[theo]{Lemma}

\newtheorem{rema}[theo]{Remark}

\newcommand{\F}{{\mathcal F}}

\newcommand{\cc}{\mathbb{C}}

\newcommand{\z}{\mathbb{Z}}

\begin{document}

\title[Biderivations and post-Lie algebra]{Super-biderivations and post-Lie superalgebras on some Lie superalgebras}

\author{ Munayim Dilxat}
\address{College of Mathematics and System Sciences, Xinjiang University, Urumqi 830046, Xinjiang, China}
\email{munayim@stu.xju.edu.cn}

\author{Shoulan Gao}
\address{Department of Mathematics, Huzhou University, Zhejiang Huzhou, 313000, China}
\email{gaoshoulan@zjhu.edu.cn}

\author{Dong Liu}
\address{Department of Mathematics, Huzhou University, Zhejiang Huzhou, 313000, China}
\email{liudong@zjhu.edu.cn}

\date{}
\maketitle

\begin{abstract}In this paper, all symmetric super-biderivations of some Lie superalgebras are determined.
  As an application, commutative post-Lie superalgebra structures on these Lie superalgebras are also obtained.
\end{abstract}


{\bf \scriptsize MR(2000) Subject Classification} {\scriptsize 17B05, 17B40, 17B68, 17B70}     

\section{Introduction}

Lie superalgebras as a generalization of Lie algebras came from supersymmetry in mathematical physics. The theory of Lie superalgebras plays key roles in modern mathematics and physics. Derivations and generalized derivations are very important subjects in the research of both algebras and their generalizations. In recent years,  many authors put so much effort into the problems of biderivations \cite{DB,B3,BZ,C,C2, XDC,TXM,TL,XTB,WY,WYC,CCC,ZFL}. The concept of biderivations was introduced and studied in \cite{B3}, and the author showed that all biderivations on commutative prime rings are inner biderivations  and then determined the biderivations of semiprime rings. The skew-symmetric biderivations of Lie algebras and Lie superalgebras have been sufficiently studied. In \cite{BZ, TMC}, the authors proved that all skew-symmetric biderivations on any perfect and centerless Lie algebras or Lie superalgebras are inner super-biderivations. With this result, all skew-symmetric biderivations on many Lie (super)algebras, which are related to the Virasoro algebra, can be determined (see \cite{BZ} for details).

It is well known that any biderivation can be decomposed into a skew-symmetric biderivation and a symmetric biderivation, and the later can determined commutative post-Lie algebra structures, which are connected with the homology of partition posets and Koszul operads \cite{bvh}.
However, there is no any sufficient tool to determine symmetric biderivations on some Lie algebras and Lie superalgebras.
Motivated by \cite{F} in calculating cohomology groups of the Virasoro algebra, we develop a general method to determine all symmetric biderivations on some Lie algebras and Lie superalgebras related to the Virasoro algebra in this paper.
We first study the biderivations from the Virasoro algebra to its density modules and then extend it to the super case. Based on these results, we determine all symmetric biderivations of some Lie superalgeras. As an application, commutative post-Lie algebra structures on these Lie superalgebras are obtained.

Throughout  this paper, we denote by $\mathbb{C}, \mathbb{Z}$ and $\mathbb{Z}^*$ the set of complex numbers,  the set of integers and nonzero integers,  respectively. All vector spaces are based on $ \mathbb{C}$.

\section{Preliminaries}

Let $L$ be a Lie superalgebra. We use the notation $|x| \ (\overline{0} $ or $ \overline{1}) $ to denote the $\mathbb{Z}_{2}$-degree of a homogeneous element $x\in L$. Elements in $L_{\overline{0}}$ or $L_{\overline{1}}$ are called homogeneous
   whenever $|x|$ is written.

\begin{defi} Let $L$ be a Lie superalgebra and $M$ be an $L$-module. A homogeneous bilinear map $\varphi: L \times L \rightarrow M$ is called a super-biderivation from $L$ to $M$ if it satisfies the following  equations:
 \begin{eqnarray}
 && \varphi([x, y], z)= (-1)^{|\varphi||x|}x\cdot\varphi(y, z)-(-1)^{|y|(|\varphi|+|x|)}y\cdot\varphi(x, z), \label{2.111}\\
&&\varphi(x, [y, z])=(-1)^{(|\varphi|+|x|)|y|} y\cdot\varphi(x, z)-(-1)^{|z|(|\varphi|+|x|+|y|)} z\cdot\varphi(x, y)  \label{1.10}
 \end{eqnarray}
   for all  $x, y, z \in L$.
\end{defi}

\begin{rema}\label{center}Clearly, if $M$ is trivial $L$-module and then for any bidiervation $\varphi: L \times L \rightarrow M$ we have $\varphi(L, [L, L])=0$ by (\ref{1.10}).
\end{rema}

Denote by ${\rm BDer}_{\gamma}(L, M)$ the set of all super-biderivations of homogenous $\gamma$ from $L$ to $M$, and ${\rm BDer}(L, M)$ the set of all super-biderivations from $L$ to $M$.
 Obviously, ${\rm BDer}(L, M)={\rm BDer}_{\overline{0}}(L, M)+{\rm BDer}_{\overline{1}}(L, M)$.
 If $M=L$, set ${\rm BDer}(L):={\rm BDer}(L, L)$. In this case a super-biderivation from $L$ to $L$ is also called a super-biderivation of $L$.

Let $L$ be a Lie superalgebra, $M$ be an $L$-module. We define Cent$_L(M)$ as follows:
A linear map $\gamma:L\to M$ belongs to Cent$_L(M)$ if
$$\gamma([x,y]) =x\cdot\gamma(y),\quad   \forall x, y\in L.$$

If  $M=L$ and $x\cdot y=[x,y]$, this is the usual centroid of $L$.

We call a super-biderivation $\varphi$  is skew-symmetry if
\begin{align}\label{a1}
 \varphi(x, y)= -(-1)^{|x||y|}\varphi(y,x), \ \forall  x,y,z \in L;
\end{align}
symmetry if
\begin{align}\label{b1}
 \varphi(x, y)= (-1)^{|x||y|}\varphi(y,x), \ \forall  x,y,z \in L.
 \end{align}
Clearly, every super-biderivation involves skew-symmetric and symmetric parts.

For skew-symmetric biderivations, some general results were first given by \cite{BZ} for Lie algebras and extended to Lie superalgebras in \cite{TMC} recently:
\begin{theo}(\cite{BZ, TMC})
Let $L$ be a perfect Lie superalgebra  and $M$ be an  $L$-module such that $Z_M(L):=\{u\in M\mid x\cdot u=0, \forall x\in L\}=\{0\}$.
Then every skew-symmetric biderivation $\delta:L\times L\to M$  is of the form
$\delta(x,y) = \gamma([x,y])$, where $\gamma\in {\rm Cent}_L(M)$.
\end{theo}

However, there are no sufficient tools to calculate symmetric (super) biderivations over Lie (super)algebras. In this paper, we mainly study symmetric biderivations on some Lie algebras and symmetric super-biderivations on some Lie superalgebras related to the Virasoro algebra.

\begin{defi}
The Virasoro algebra $ {\rm Vir}={\rm span}_{\mathbb{C}}\{ L_{m}, C \mid m\in \mathbb{Z}\}$ is  the Lie algebra equipped with the following relations:
\begin{align}\label{1c}
[L_{m}, L_{n}]= (m-n)L_{m+n}+\frac{1}{12}\delta_{m+n, 0}(m^{3}-m)C,
\end{align}
\begin{align}\label{1K}
 [L_{m}, C]=0, \ \forall \ m, n \in \mathbb{Z}.
  \end{align}
\end{defi}

\begin{defi}
The so-called density module $\F_b$ over the Virasoro algebra for some $b\in\mathbb C$ is  $\F_{b}:=\sum_{i\in\mathbb Z}\mathbb Cv_i$ with
\begin{align*}L_mv_i=-(i+b m)v_{m+i}, \,  Cv_i=0, \quad \forall m, i\in\mathbb Z.
\end{align*}
\end{defi}

With the density module $\F_b$,  we can get a Lie algebra $W(0, b):={\rm Vir}\ltimes \F_b$  (see \cite{GS}). More details,
 for $b\in\mathbb C$, the Lie algebra $W(0, b)$ is  the Lie algebra $ \{ L_{m}, I_m,  C \mid m\in \mathbb{Z}\}$  equipped with $(\ref{1c})$ and the following relations:
       \begin{align}\label{1.2}
       &[L_{m}, I_{n}]= -(bm+n)I_{m+n},   \\
    \label{1.3}
    & [I_{m}, I_{n}]= 0, [x, C]= 0
         \end{align}
for any $m, n \in \mathbb{Z}, \  x\in W(0, b)$.

\section{Biderivations from the Virasoro algebra to its density modules}

In this section,  we mainly determine all symmetric biderivations on the Virasoro algebra {\rm Vir} to its density module $\F_b$.

Define a  linear map $\epsilon: {\rm Vir}\rightarrow\F_b$ as follows:
\begin{align}\label{Gdef}
\epsilon(L_{m})=v_{m}, \, \epsilon(C)=0, \quad\, \forall\,m\in\mathbb{Z}.
\end{align}
It is easy to check that $\epsilon\in {\rm Cent_{Vir}}(\F_b)$.
\begin{lemm}\label{cen}
 For the Virasoro algebra {\rm Vir} and its density module $\F_b$,
\begin{align}\label{3a}
{\rm Cent_{Vir}}(\F_b)=\begin{cases}
\mathbb C\epsilon, & if \  b=-1,\\
0,& otherwise.
\end{cases}
\end{align}
\end{lemm}
\begin{proof}
Let $\gamma\in {\rm Cent_{Vir}}\F_b$. Without loss of generality,  we can suppose that
$\gamma(L_{m})=\sum_{k\in\mathbb{Z}}a_{m,k}v_{m+k}$ and
$\gamma(C)=\sum_{k\in\mathbb{Z}}c_{k}v_{k}$  for some $a_{m, k}, c_{k}\in\mathbb C$.
 From $ \gamma([L_m, L_n])=L_{m}.\gamma(L_{n}) $, we get
 \begin{align}\label{3b}
  (m-n)a_{m+n,k}=(-bm-n-k)a_{n,k}.
   \end{align}
  Taking $m=0$ in  \eqref{3b}, we get $a_{n,k}=0$ if $k \neq 0$, and then
  $\gamma(L_{m})=a_{m,0}v_{m}$.
  Taking $n=k=0$ in  \eqref{3b}, we get $a_{m,0}=-ba_{0,0}$. By \eqref{3b},  we get
 $b(b+1)a_{0,0}=0$.
  Thus, we get $a_{m,0}=a_{0,0}\in \mathbb{C}$ if $b=-1$, and $a_{m,0}=0$ if $b\ne -1$ for any $m\in\mathbb Z$.

 Since $\gamma[L_{0}, C]=L_{0}\cdot \gamma(C)$, we have
$$0=L_{0}\cdot \gamma(C)=L_{0}\cdot\sum c_k v_{k}=\sum c_{k} L_{0}\cdot v_{k}=-\sum kc_{k} v_{k}.$$
Then $kc_{k}=0$, which forces $c_{k}=0$ for $k\neq 0$. Thus, we get
$\gamma(C)=c_{0}v_{0}$.
Moreover, by $\gamma[L_{m}, C]=L_{m}\cdot \gamma(C)$,
 we have
$$0=L_{m}\cdot c_{0}v_{0}=c_{0}mv_{m}.$$
So $c_{0}m=0$ for all $m\in\mathbb{Z}$, which suggests  $c_{0}=0$.  The lemma holds.
\end{proof}

With Theorem 2.3 in \cite{BZ} and Lemma \ref{cen}, we have the following result.

\begin{theo}\label{000}
 Every skew-symmetric biderivation $\delta\in {\rm BDer}({\rm Vir}, \F_{b})$ is as follows:
\begin{align}
&\delta(L_{m},L_{n})=\lambda\epsilon([L_m, L_n])=\begin{cases}
\lambda(m-n)v_{m+n}, & if \  b=-1,\\
0,& otherwise.
\end{cases}\\
&\delta(x,C)=0
\end{align}
for any $ m,n\in \mathbb{Z}, x\in{\rm Vir}, \lambda\in \mathbb{C }$.
\end{theo}

Now we shall determine all biderivations (no skew-symmetry required) of the Virasoro algebra {\rm Vir} to $\F_{b}$, which is very helpful to do such researches for some Lie (super)algebras related to the Virasoro algebra.

  \begin{theo}\label{main}
Every symmetric biderivation $\delta\in {\rm BDer}({\rm Vir}, \F_{b})$ is as follows:
  \begin{align*}
  &  \delta(L_{m},L_{n})=\begin{cases}
\sum_{k\in \mathbb{Z}}\mu _{k}\upsilon_{m+n+k}, & if  \  b=0,\\
\sum_{k\in \mathbb{Z}}(m+n+k)\mu_{k}\upsilon_{m+n+k}, & if \ b=1,\\
0,& otherwise,
\end{cases}\\
&\delta(x, C)=\delta(C, x)=0
  \end{align*}
  for any $m,n\in \mathbb{Z}, x\in {\rm Vir}, \mu_{k} \in \mathbb{C}, k\in\mathbb Z$.
   \end{theo}

 \begin{proof} Let $\delta\in {\rm BDer}({\rm Vir}, \F_{b})$.
Suppose that
$$\delta(L_{m},L_{n})=\sum_{k\in\z}a_{m,n,k}v_{m+n+k}, \, \delta(L_{m},C)=\sum_{k\in\z}c_{m,k}v_{m+k}$$ for some $a_{m, n,k},c_{m,k}\in\mathbb C$.

From
$\delta (L_{m},[L_{0},C])=L_{0}\cdot\delta(L_{m},C)-C\cdot\delta(L_{m},L_{0})$,
we get
$ \sum_{k\in\z}(m+k)c_{m,k}v_{m+k}=0$,
then $(m+k)c_{m,k}=0$, so $\delta(L_{m},C)=c_{m,-m}v_0$.
From $\delta ([L_{m},L_{0}],C)=L_{m}\cdot\delta(L_{0},C)-L_0\cdot\delta(L_{m},C)$,
we get $ mc_{m,-m}v_0=-bmc_{0,0}v_m$.
 Then  $c_{m,-m}v_0=-bc_{0,0}v_m$ if $m\neq 0$.
 So $c_{m,-m}=0$ if $m\ne0$.
 Thus, we have $\delta(L_{m},C)=\delta_{m,0}c_{0,0}v_0$.
  From $\delta ([L_{1},L_{-1}],C)=L_{1}\cdot\delta(L_{-1},C)-L_{-1}\cdot\delta(L_{1},C)$,
   we get $\delta(L_{m},C)=0$.
Thus, we get
$$\delta(L_{m},C)=0,  \,  \forall m \in \mathbb{Z}.$$
From
$$\delta ([L_{m},L_{n}],L_{p})=L_{m}\cdot\delta(L_{n},L_{p})-L_{n}\cdot\delta(L_{m},L_{p}),$$
$$\delta (L_{m},[L_{n},L_{p}])=L_{n}\cdot\delta(L_{m},L_{p})-L_{p}\cdot\delta(L_{m},L_{n}),$$
we get
\begin{align}\label{3c}
(m-n)a_{m+n,p,k}=(bn+m+p+k)a_{m,p,k}-(bm+n+p+k)a_{n,p,k},
 \end{align}
\begin{align}\label{3d}
 (n-p)a_{m,n+p,k}=(bp+m+n+k)a_{m,n,k}-(bn+m+p+k)a_{m,p,k}.
 \end{align}
 Setting $m=p=0$ in \eqref{3c}, we get
\begin{align}\label{3eee}
 ka_{n,0,k}=(bn+k)a_{0,0,k}.
 \end{align}
  If $k\neq 0$, we get
   \begin{align}\label{3e}
        a_{n,0,k}=\frac{bn+k}{k}a_{0,0,k}.
        \end{align}
 Setting $p=0$ in \eqref{3d}, we get
 \begin{align}\label{3e1}
 (m+k)a_{m,n,k}=(bn+m+k)a_{m,0,k}.
  \end{align}
  Combining with \eqref{3e}, we get
   \begin{align}\label{3e1}
   (m+k)a_{m,n,k}=(bn+m+k)\frac{bm+k}{k}a_{0,0,k}.
   \end{align}
  Set $m=-k\neq 0$ in \eqref{3e1},  then we have
  \begin{align}\label{33}
  b(b-1)a_{0,0,k}=0, \, k\neq 0.
  \end{align}

Next, we shall consider $a_{m, n, 0}$.
 From \eqref{3c}  and \eqref{3d}, we get
\begin{align}\label{33c}
(m-n)a_{m+n,p,0}=(bn+m+p)a_{m,p,0}-(bm+n+p)a_{n,p,0},
\end{align}
\begin{align}\label{33d}
(n-p)a_{m,n+p,0}=(bp+m+n)a_{m,n,0}-(bn+m+p)a_{m,p,0}.
\end{align}
 Set $p=0$ in \eqref{33d},  then we have
 \begin{align}\label{444d}
  a_{m,n,0}=\frac{bn+m}{m}a_{m,0,0}, \, m \neq 0.
  \end{align}
 By $\delta(L_{m},L_{n})=\delta(L_{n},L_{m})$, we have $a_{m,n,0}=a_{n, m, 0}$ for all $m, n\in\mathbb Z$.
  So
 \begin{align}\label{55d}
   a_{m,n,0}=\frac{bn+m}{m}a_{m,0,0}=\frac{bm+n}{n}a_{n,0,0}, \,   \,  \forall m,n\neq0.
\end{align}
 Set $m=1$ in \eqref{55d},
\begin{align}\label{k1}
(n+b)a_{n,0,0}=n(bn+1)a_{1,0,0}.
   \end{align}
 Setting $m=n+1$ in \eqref{55d} and using \eqref{k1}, we have
\begin{align}\label{k2}
(bn+n+1)(n+b)a_{n+1,0,0}=(bn+1)(n+1)(bn+n+b)a_{1,0,0}.
\end{align}
 Combining with \eqref{k1} and \eqref{k2}, we get
  \begin{align}\label{k222}
  b(b-1)a_{1,0,0}=0.
  \end{align}

 \noindent{\bf Case 1.}  $b=0$.  If $k\neq -m$ in  \eqref{3e1}, we have
  $$a_{m,n,k}=a_{m,0,k}=a_{0,0,k}, \,  \forall k\ne 0.$$
   If $k=-m$,  setting $n=0$ in \eqref{3c},
    we get
   $$a_{m,n,-m}=a_{0,n,-m}=a_{0,0,-m}, \forall m\neq n.$$
    If $k=-m=-n$, from \eqref{3d},
    we have
    $$a_{m,m,-m}=a_{m,0,-m}=a_{0,0,-m}.$$
     So
     $$a_{m,n,k}=a_{m,0,k}=a_{0,0,k}, \forall  m,n \in \mathbb{Z}, \, k\ne 0.$$
   By \eqref{55d}, we have
    $$a_{m,n,0}=a_{m,0,0}=a_{1,0,0}, \forall  m,n \in \mathbb{Z}.$$
  Thus, we have
$$\delta(L_{m},L_{n})=\sum_{k\in \mathbb{Z}}\mu _{k}\upsilon_{m+n+k},$$
where $\mu_{0}=a_{1,0,0}$ and $\mu_{k}=a_{0,0,k}$ for all $k\neq0$.

 \noindent{\bf Case 2.} $b=1$. Setting $n=p=0$ in  \eqref{3c}, we get
  $$a_{m,0,k}=\frac{m+k}{k}a_{0,0,k}, \forall  m \in \mathbb{Z},k\neq 0.$$
  Setting $p=0$ in \eqref{3d},
   we get
  $$ a_{m,n,k}=(m+n+k)\frac{a_{0,0,k}}{k}, \forall k\neq -m \ne0.$$
If $k=-m$, from \eqref{3c} with $n=0$,
 we get
$$ a_{m,n,-m}=\frac{n}{-m}a_{0,0,-m}, \forall m\neq  n\neq 0.$$
If $k=-n=-m$, from \eqref{3d},
we get
$$a_{m,m,-m}=-a_{0,0,-m}.$$
So
 $$ a_{m,n,k}=(m+n+k)\frac{a_{0,0,k}}{k}, \forall m,n\in \mathbb{Z}, k \ne0.$$
By \eqref{55d},
we have
$$a_{m,n,0}=(m+n)a_{1,0,0}.$$
 Thus, we have
  $$\delta(L_{m},L_{n})=\sum_{k\in \mathbb{Z}}(m+n+k)\mu_{k}\upsilon_{m+n+k},$$
   where $\mu_k=\frac{1}{k}a_{0,0, k}$ for $k\ne 0$ and $\mu_0=a_{1, 0, 0}$.

 \noindent{\bf Case 3.}  $b\neq 0, 1$.  By \eqref{444d} and \eqref{k1},  we get $a_{m, n, 0}=0$ for all $m, n, k\in\mathbb Z$. Combining with \eqref{33}, we can get
$\delta(L_{m},L_{n})=0$.
\end{proof}

With Theorem \ref{main}, we can easily get all biderivations on some Lie algebras related to the Virasoro algebra.

  \begin{coro}\label{vir}
  Every symmetric biderivation of {\rm Vir} is trivial.
  \end{coro}
\begin{proof}
 Clearly,  ${\rm Vir}/\cc C\cong \F_{-1}$. So it follows directly by Theorem \ref{main} and Remark \ref{center}.\end{proof}

\begin{coro} \label{3.5}
  Let $\delta$ be a symmetric biderivation of $W(0, b)$, then
 \begin{align*}
&\delta(L_{m},L_{n})=\begin{cases}
\sum_{k\in \mathbb{Z}}\mu _{k}I_{m+n+k}, & if  \  b=0,\\
\sum_{k\in \mathbb{Z}}(m+n+k)\mu_{k}I_{m+n+k}, & if \ b=1,\\
0 ,& otherwise,
\end{cases}\\ \label{www2}
&\delta(L_{m},I_{n})=\delta(I_{m},I_{n})=\delta(x, C)=0
 \end{align*}
for any $m,n\in \mathbb{Z}$, $x\in W(0, b)$, where $\mu_{k} \in \mathbb{C}$.
\end{coro}
\begin{proof}  ${\rm BDer}({\rm Vir}, W(0, b))={\rm BDer}({\rm Vir}, {\rm Vir})+{\rm BDer}({\rm Vir}, \F_b)$.
 By Theorem \ref{main} and Corollary \ref{vir}, we can easily get the conclusion.
\end{proof}


\section{ Biderivations from the super Virasoro algebra to its density modules}

Now we introduce the so-called density module $\mathfrak F_b$ over the super
Virasoro algebra for some $b\in\mathbb C$.

\begin{defi}
 The super Virasoro algebra ${\rm SVir}$ (Ramond sector) is a Lie superalgebra spanned by
  $ \{ L_{m}, G_{m}, C \mid m\in \mathbb{Z}\}$, equipped with the following relations:
\begin{align*}
& [L_{m}, L_{n}]= (m-n)L_{m+n}+\frac{1}{12}\delta_{m+n, 0}(m^{3}-m)C, \\
&[L_m,G_r]=(\frac{m}{2}-r)G_{m+r},\\
& [G_{r}, G_{s}]= 2L_{r+s}+\frac{1}{3}\delta_{r+s, 0}(r^{2}-\frac{1}{4})C,\\
&[G_{r}, C]=[L_{m}, C] =0,\, \forall m, n, r, s\in \mathbb{Z}.
\end{align*}
\end{defi}

Clearly, the even part ${\rm SVir}_{\bar0}$ is the Virasoro algebra. Some structures and representations was studied in \cite{DGL, CLL, GMP, LPX1, LPX2, S}, etc.

\begin{defi} \cite{S, GMP}
For $b\in\mathbb{C}$, the density module $\mathfrak F_b$ over ${\rm SVir}$ is the vector space
$$\mathfrak F_b=\bigoplus_{n\in\mathbb{Z}}\mathbb{C}I_{n}\bigoplus\bigoplus_{r\in
\mathbb{Z}}\mathbb{C}J_{r}$$ is a ${\rm SVir}$-supermodule with the parity
$$ |I_{n}|=\overline{0},     \      |J_{r}|=\overline{1},$$
satisfying the following commutation relations:
 \begin{align*}
     &L_{m}\cdot I_{n} =-(n+bm)I_{m+n},\\
     &L_{m}\cdot J_{r} =-(r+(b+\frac{1}{2})m)J_{m+r},\\
  &G_{r}\cdot I_{m} =-(\frac{m}{2}+br) J_{m+r},    \\
   &G_{r}\cdot J_{s}=2I_{r+s},\quad  C\cdot\mathfrak F_b=0
           \end{align*}
for all $m,n, r, s \in\mathbb{Z}$.
\end{defi}

Define a  linear map $\varepsilon: {\rm SVir}\rightarrow\mathfrak F_b$ as follows:
\begin{equation}\label{Gdef}
\varepsilon(L_{m})=I_{m}, \varepsilon(G_{m})=J_{m}, \varepsilon(C)=0, \quad\, \forall\,m\in\mathbb{Z}.
\end{equation}
It is easy to check that $\varepsilon\in {\rm Cent_{SVir}}(\mathfrak F_b)$.

\begin{lemm} \label{scen}
For the  super Virasoro algebra {\rm SVir} and its density module $\mathfrak F_b$, we have
\begin{align}
{\rm Cent_{{\rm SVir}}}(\mathfrak F_b)=\begin{cases}
\mathbb C\varepsilon, & if \  b=-1,\\
0,& otherwise.
\end{cases}
\end{align}
\end{lemm}
\begin{proof}
Let  $\gamma\in {\rm Cent_{SVir}}(\mathfrak F_b)$.
Without loss of generality,  we can suppose that $ \gamma(L_{m})=\sum_{k\in \mathbb{Z}}a_{m,k}I_{m+k}$ and
 $ \gamma(G_{r})=\sum_{k\in\mathbb{Z}}b_{r,k}J_{r+k}$  for some  $a_{m, k},b_{r,k} \in\mathbb C$.

From $ \gamma([L_m, L_n])=L_{m}\cdot\gamma(L_{n})  $, we get
 \begin{align}\label{4a}
  (m-n)a_{m+n,k}=(-bm-n-k)a_{n, k}.
  \end{align}
  Setting $m=0$ in \eqref{4a}, we have $a_{n, k}=0$ for $k\neq 0$. For $k=0$,  \eqref{4a} becomes
  \begin{align}\label{4a-1}
  (m-n)a_{m+n,0}=(-bm-n)a_{n, 0}, \,  \forall m,n\in\mathbb Z.
  \end{align}
  Setting $n=0 $ in \eqref{4a-1}, we have $a_{m,0}=-b a_{0,0}$ for $m \neq 0$.
  Then by \eqref{4a-1}, we get
  $$b(b+1)a_{0,0}=0, \,   \forall m \neq 0.$$
  If $b=-1$, then $a_{m,0}=a_{0,0}$.
   If $b=0$, then $a_{m,0}=a_{0,0}=0$.
   If $b\neq-1,0$, then $a_{0,0}=0$.

 From $ \gamma([L_m, G_r])=L_{m}\cdot\gamma (G_{r} )$,
  we get
 \begin{align}\label{4b}
 (\frac{m}{2}-r)b_{m+r, k}=-(k+r+(b+\frac{1}{2})m)b_{r, k}.
 \end{align}
 Setting $m=0$ in \eqref{4b}, we have
 $ b_{r, k}=0$ for $k\neq 0$.  For $k=0$,  \eqref{4b} becomes
  \begin{align}\label{4b-1}
 (\frac{m}{2}-r)b_{m+r, 0}=-(r+(b+\frac{1}{2})m)b_{r, 0}, \,  \forall m,r \in  \mathbb Z.
 \end{align}
 Setting $r=0 $ in \eqref{4b-1}, we have $b_{m,0}=-(2b+1) b_{0,0}$ for $m \neq 0$.
  Then by \eqref{4b-1}, we get
  $$(b+1)(2b+1)b_{0,0}=0, \,   \forall m \neq 0.$$
  If $b=-1$, then $b_{m,0}=b_{0,0}$.
   If $b=-\frac{1}{2}$, then $b_{m,0}=b_{0,0}=0$.
   If $b\neq-\frac{1}{2},0$, then $b_{0,0}=0$.

From $ \gamma([G_r, G_s])=-G_{r}\cdot\gamma (G_{s})$,
we get $\label{(1)} a_{r+s,0}=-b_{s,0}$.

 \end{proof}

By Theorem 3.8 in \cite{TMC} and Lemma \ref{scen} we can get the following result.

\begin{theo}
 Every skew-symmetric biderivation $\delta\in {\rm BDer}({\rm SVir}, \mathfrak F_b)$ is as follows:
\begin{align}
&\delta(x, y)=\begin{cases}
\lambda\varepsilon([x, y]), & if \  b=-1,\\
0,& otherwise
\end{cases}
\end{align}
for all  $x, y\in{\rm SVir}$,  where $\lambda\in \mathbb{C }$.
\end{theo}

 \begin{theo}\label{main2}
  Let $\delta$ be a symmetric biderivation $\delta: {\rm SVir} \times {\rm SVir} \rightarrow \mathfrak F_b$,
  then
 $$ \delta(x,y)=0,   \, \forall x, y \in{\rm SVir}.$$
 \end{theo}

 \begin{proof}
  Let $\delta$ be a symmetric biderivation $\delta: {\rm SVir} \times {\rm SVir} \rightarrow \mathfrak F_b$. Then $\delta$ can be regarded as a symmetric biderivation $\delta: {\rm Vir} \times {\rm Vir} \rightarrow \mathfrak F_b=\F_b+\F_{b+\frac12}$.

 Clearly, ${\rm BDer}({\rm Vir}, \mathfrak F_b)={\rm BDer}({\rm Vir}, \F_b)+{\rm BDer}({\rm Vir}, \F_{b+\frac12})$.
So by Theorem \ref{main}, we get
 \begin{align}\label{ab11}
 &&\delta(L_{m},L_{n})=\begin{cases}
\sum_{k\in \mathbb{Z}}\mu _{k}I_{m+n+k}, & if  \  b=0,\\
\sum_{k\in \mathbb{Z}}(m+n+k)\mu_{k}I_{m+n+k}, & if \ b=1,\\
\sum_{k\in \mathbb{Z}}\mu _{k}J_{m+n+k},& if \ b=-\frac{1}{2},\\
\sum_{k\in \mathbb{Z}}(m+n+k)\mu _{k}J_{m+n+k}, & if \ b=\frac12,\\
0 ,& otherwise.
\end{cases}
\end{align}
Suppose that
$$\delta(L_{m}, L_n)=\sum_{k\in\z}a_{m,n, k}I_{m+n+k}+\sum_{k\in\z}b_{m,n, k}J_{m+n+k},$$
 $$\delta(L_{m},G_{r})=\sum_{k\in\z}c_{m,r, k}I_{m+r+k}+\sum_{k\in\z}d_{m,r, k}J_{m+r+k}$$
 for some $a_{m, n, k}, b_{m, n, k}, c_{m, r, k}, d_{m, r, k} \in\mathbb C$ and $a_{m, n, k}, b_{m, n, k}$ are given in (\ref{ab11}).

\noindent{\bf Step 1.} First we only consider the case of $b=\pm\frac12$.  In this case $|\delta|=1$ and $a_{m, n, k}=0$ for all $m, n, k \in\z$. We shall prove that $b_{m, n, k}=0$ for all $m, n,k \in\z$.

From
$$\delta ([L_{m},L_{n}],G_{r})=L_{m}\cdot\delta(L_{n},G_{r})-L_{n}\cdot\delta(L_{m},G_{r}),$$
$$\delta (L_{m},[L_{n},G_{r}])=L_{n}\cdot\delta(L_{m},G_{r})+G_{r}\cdot\delta(L_{m},L_{n}),$$
we get
\begin{align}\label{b}
 (m-n)c_{m+n,r,k}=(bn+m+r+k)c_{m,r,k}-(bm+n+r+k)c_{n,r,k},
  \end{align}
\begin{align}\label{c} (m-n)d_{m+n,r,k}=((b+\frac{1}{2})n+m+r+k)d_{m,r,k}-((b+\frac{1}{2})m+n+r+k)d_{n,r,k},
\end{align}
\begin{align}\label{d}
(\frac{n}{2}-r)c_{m,n+r,k}=-(bn+m+r+k)c_{m,r,k}+2b_{m,n,k},
\end{align}
\begin{align}\label{e}
(\frac{n}{2}-r)d_{m,n+r,k}=-((b+\frac{1}{2})n+m+r+k)d_{m,r,k}.
\end{align}
Setting $n=0$ in \eqref{d},
we get  $(m+k)c_{m,r,k}=2b_{m,0,k}$ and then
\begin{equation} c_{m, r, k}=c_{m, 0, k}, \forall m, r, k\in\mathbb Z, m+k\ne0.\label{d00}\end{equation}
Setting $n=0$ in \eqref{b},
 we get
 \begin{align}\label{bb}
 (r+k)c_{m, r, k}=(r+bm+k)c_{0, r, k}.
 \end{align}

 Replacing $r$ by $r+1$ in \eqref{bb}, and combing with \eqref{bb} and \eqref{d00}, we get
 \begin{equation}c_{m,r,k}=c_{0, 0, k}, \forall m, r, k\in\mathbb Z, m+k\ne 0. \label{d01}\end{equation}

  Using \eqref{b}, we get $c_{m, r, k}=0$ for any $m, r, k\in\mathbb Z$. By \eqref{d} we get
$b_{m, n, k}=0$ for all $ m, n, k\in\mathbb Z$, and then
 $\delta(L_m, L_n)=0$ for all $m, n\in\mathbb Z$.

Now we can suppose that
 \begin{align}\label{dell}
\delta(L_{m},L_{n})=\begin{cases}
\sum_{k\in \mathbb{Z}}\mu _{k}I_{m+n+k}, & if  \  b=0,\\
\sum_{k\in \mathbb{Z}}(m+n+k)\mu_{k}I_{m+n+k}, & if \ b=1,\\
0 ,& otherwise.
\end{cases}
\end{align}

\noindent{\bf Step 2.} Next we shall prove that $\delta(L_{m},G_{r})=\delta(L_{m},L_{n})=0$ for all $m, n, r\in\mathbb Z$.

In this case \eqref{d} becomes
\begin{align}\label{d1}
(\frac{n}{2}-r)c_{m,n+r,k}=-(bn+m+r+k)c_{m,r,k}.
\end{align}
 Setting $m=n=0$ in \eqref{d1}, we get $c_{0,r,k}=0$ for all $k\ne 0$. So  $c_{m,r,k}=0$ for all $r+k\ne0, k\ne0$. Using \eqref{d1} again, we get
$c_{m,r,k}=0$ for all $k\ne0$.

 Setting $n=k=0$ in \eqref{d1}, we get $c_{m,r,0}=0$ for all $m\ne 0$. Combining with \eqref{bb}, we get $c_{m,r,0}=0$ for all $r\in \mathbb{Z}^*$. Using \eqref{d1} again we get $c_{m,r,0}=0$ for all $m, r\in\mathbb{Z}$.

So we get $\delta(L_{m},G_{r})=\sum_{k\in\mathbb{Z}}d_{m,r, k}J_{m+r+k}$ for all $m, n\in\mathbb{Z}$. In this case $|\delta|=0$.

From
$$\delta (L_{m},[L_{n},G_{r}])=L_{n}\cdot\delta(L_{m},G_{r})-G_{r}\cdot\delta(L_{m},L_{n}),$$
we get
\begin{align}\label{f} (\frac{n}{2}-r)d_{m,n+r,k}=-((b+\frac{1}{2})n+m+r+k)d_{m,r,k}+(br+\frac{m+n+k}{2})a_{m,n,k},
\end{align}
 where $a_{m, n, k}, m, n, k\in\mathbb Z$ are given by \eqref{dell}.

Setting $n=0$ in \eqref{f}, we get
\begin{align}\label{f11}
 (m+k)d_{m,r,k}=({br}+\frac{1}{2}(m+k))a_{m,0,k}.
 \end{align}

Taking $m=n=0$ in \eqref{f}, we get
 \begin{align}\label{f12} kd_{0,r,k}=(br+\frac{1}{2}k)a_{0,0,k}.
 \end{align}

Setting $n=0$ in \eqref{c}, we get
 \begin{align}\label{f13}
 (r+k)d_{m, r, k}=((b+\frac12)m+r+k)d_{0, r, k}.
 \end{align}

If $b=0$, then $a_{m, n, k}=\mu_k$ for all $m, n, k\in\mathbb{Z}$ by \eqref{dell}.  Combining with \eqref{f11}-\eqref{f13}, we get $d_{m, n, k}=0$ for all $m, n, k\in\mathbb{Z}$.

If $b=1$, then $a_{m, n, k}=(m+n+k)\mu_k$ for all $m, n, k\in\mathbb{Z}$ by \eqref{dell}.  Combining with \eqref{f11}-\eqref{f13}, we can also get $d_{m, n, k}=0$ for all $m, n, k\in\mathbb{Z}$.

If $b\ne0, 1$, it is clear that $d_{m, n, k}=0$ for all $m, n, k\in\mathbb{Z}$. In this case,
$d_{m,r,k}=a_{m,n,k}=\mu_{k}=0$  for all $m,n,r,k\in\mathbb{Z} $,  then we get
$$\delta(L_{m},G_{r})=\delta(L_{m},L_{n})=0, \forall m,n, r\in\mathbb{Z}.$$

\noindent{\bf Step 3.} We shall determine $\delta(G_{r},G_{s})$ with step $2$.

Suppose that
$\delta(G_{r},G_{s})=\sum_{k\in\z}e_{r,s,k}I_{s+r+k}$ for some $e_{ r,s,k}\in\mathbb C$.
From
$$\delta ([L_{m},G_{r}],G_{s})=L_{m}\cdot\delta(G_{r},G_{s})+G_{r}\cdot\delta(L_{m},G_{s}),$$
$$\delta (G_{r},[G_{s},G_{t}])=-G_{s}\cdot\delta(G_{r},G_{t})-G_{t}\cdot\delta(G_{r},G_{s}),$$
we get
\begin{align}\label{11}
 (\frac{m}{2}-r)e_{m+r,s,k}+(r+s+k+bm)e_{r,s,k}=0,
  \end{align}
\begin{align}\label{61}
 (\frac{r+t+k}{2}+bs)e_{r,t,k}+(\frac{r+s+k}{2}+bt)e_{r,s,k}=0.
 \end{align}
 Setting $m=0$ in \eqref{11}, we get $ (k+s)e_{r,s,k}=0$. Setting $s=t=0$ in \eqref{61}, we get $(r+k)e_{r,0,k}=0$, then $e_{r,s,k}=e_{r,0,k}=e_{0,0,k}=0$ with $k\neq 0$.
  If $k=0$, we have $e_{r,s,0}=e_{r,0,0}=0$.
So
$\delta(G_{r},G_{s})=0$ for all $r, s\in\mathbb{Z}$.

\noindent{\bf Step 4.} Finally we can determine $\delta(G_{r},C)$ with step $2$.

Suppose that
$\delta(G_{r}, C)=\sum_{k\in\mathbb{Z}}f_{r,k}I_{r+k}+\sum_{k\in\mathbb{Z}}g_{r,k}J_{r+k}$
 for some $f_{r,k}, g_{r,k} \in\mathbb C$. From
$$\delta ([C,L_0],G_r)=C\cdot\delta(L_{0},G_{r})-L_{0}\cdot\delta(C,G_{r}),$$
we get
$(r+k)f_{r,k}I_{r+k}+(r+k)g_{r,k}J_{r+k}=0$, then $f_{r,k}=g_{r,k}=0$ if $k\neq -r$. If $k=-r$, from
$\delta ([L_1,G_{r}],C)=L_{1}\cdot\delta(G_{r},C)-G_{r}\cdot\delta(L_{1},C)$,
we get $f_{r,-r}=g_{r,-r}=0$.
Then we get $\delta({\rm SVir}, C)=0$.
\end{proof}

\begin{rema}
By a similar calculation, we can get that Theorem \ref{main2} also holds for the super Virasoro algebra of the Neveu-Schwarz sector.
\end{rema}

\section{Biderivations on some Lie superalgebras}

In this section, we shall determine all symmetric biderivations on some Lie superalgebras with Theorem \ref{main} and Theorem \ref{main2}.

\subsection{The super Virasoro algebra}

\begin{theo}
 Every  symmetric super-biderivation of ${\rm SVir}$ is trivial.
\end{theo}
\begin{proof}
 It follows by Theorem \ref{main2} and Remark \ref{center}.
  \end{proof}
\begin{rema}
 All skew-symmetric biderivation on the super Virasoro algebra {\rm SVir} were determined in \cite{FD}.
\end{rema}

\subsection{The super  ${\rm W(2,2)}$ algebra}

In this subsection, we shall determine symmetric biderivations  of the super ${\rm W(2,2)}$ algebra ${\rm SW}$.

By definition, the  super ${\rm W(2,2)}$ algebra ${\rm SW}$ is a Lie superalgebra over $\cc$ with a basis $\{L_m, H_m, G_r, Q_{r},C_{1},C_{2}\mid m\in\mathbb{Z},r\in\mathbb{Z}+\frac{1}{2}\}$ and the following relations:
\begin{align*}
& [L_m,L_n]=(m-n)L_{n+m}+{1\over12}\delta_{m+n,0}(m^3-m)C_{1},\\
&[L_m, H_n]=(m-n)H_{m+n}+{1\over12}\delta_{m+n,0}(m^3-m)C_{2},\\
&[L_m,G_r]=(\frac{m}{2}-r)G_{m+r},\quad  [L_m,Q_r]=(\frac{m}{2}-r)Q_{m+r},\\
&[G_r,G_s]=2L_{r+s}+{1\over3}\delta_{r+s,0}(r^{2}-\frac{1}{4})C_{1},\\
&[G_r,Q_s]=2H_{r+s}+{1\over3}\delta_{r+s,0}(r^{2}-\frac{1}{4})C_{2},\\
&[H_m,G_r]=(\frac{m}{2}-r)Q_{m+r}, \quad  [x,C_{1}]=[x,C_{2}]=0
\end{align*}
for any $m,n\in\mathbb{Z}, r, s\in\mathbb{Z}+\frac{1}{2}, x\in {\rm SW}$.

${\rm SW}$  has the decomposition:
$SW=SW_{\bar0}\oplus SW_{\bar1},$
where$$
 SW_{\bar0}=\bigoplus_{n\in\z}\cc L_n\oplus\bigoplus_{n\in\z}\cc H_n\oplus\cc  C_{1}\oplus\cc  C_{2},\quad SW_{\bar1}=\bigoplus_{r\in\z+\frac12} \cc  G_r\oplus\bigoplus_{s\in\z+\frac12}\cc Q_{r}.$$
\begin{theo}
Every  symmetric super-biderivation of the super $W(2,2)$ algebra  ${\rm SW}$ is trivial.
\end{theo}
\begin{proof}
Let $\phi$ be a super-biderivation of the super $W(2,2)$.
From  Theorem \ref{main}, Theorem \ref{main2} and Remark \ref{center}, we know
 $\phi(L_{m},L_{n})=\phi(L_{m},H_{n})= \phi(H_m,H_{n})=0$, $\phi(L_{m},C_{1})
 =\phi(L_{m},C_{2})=0$, $\phi(H_{n},C_{1})=\phi(H_{n},C_{2})=0$, $\phi(G_{r},C_{1})=\phi(G_{r},C_{2})=0$,
 $\phi(L_m,G_r)=\phi(L_m,Q_r)= \phi(G_r,G_{s})=0$ for all  $m, n\in\mathbb{Z}, r, s\in\mathbb{Z}+\frac12$.
So we need to determine $\phi(H_{m},G_{r})$, $\phi(G_{r},Q_{s})$, $\phi(Q_{r},Q_{s})$, $\phi(H_{m},Q_{r})$, $\phi(Q_{r},C_{1})$ and $\phi(Q_{r},C_{2})$.

 Assume that
 $$\phi(H_m,G_r)=\sum_{\alpha\in\mathbb{Z}+\frac{1}{2}}a_{m,r,\alpha}^{(1)}Q_{\alpha}+
 \sum_{\beta\in\mathbb{Z}+\frac{1}{2}}b_{m,r,\beta}^{(1)}G_{\beta}$$
 for some $a_{m, r,\alpha}^{(1)},b_{m, r,\beta}^{(1)}\in\mathbb{C}.$
By
$$\phi(H_{m},[G_{r},G_{r}])=[\phi(H_{m},G_{r}), G_{r}]+(-1)^{(|\phi|+|H_{m}|)|G_{r}|}[G_{r},\phi(H_{m},G_{r})],$$
 we get
 $$ \sum_{\alpha\in\mathbb{Z}+\frac{1}{2}}a_{m,r,\alpha}^{(1)}H_{\alpha+r}
 +\sum_{\beta\in\mathbb{Z}+\frac{1}{2}}b_{m, r,\beta}^{(1)}L_{\beta+r}=0.$$
  Thus, $\phi(H_m,G_r)=0$.

By Remark \ref{center} we can assume that
$$ \phi(G_r,Q_{s})=\sum_{i\in\mathbb{Z}}a_{r,s,i}^{(2)}L_{i}+
 \sum_{j\in\mathbb{Z}}b_{r,s,j}^{(2)}H_{j}$$
 for some $a_{r,s,i}^{(2)},b_{r,s,j}^{(2)}\in\mathbb{C}.$
 By
  $$\phi(G_{r},[Q_{s},L_{0}])=[\phi(G_{r},Q_{s}), L_{0}]+(-1)^{(|\phi|+|G_{r}|)|Q_{s}|}[Q_{s},\phi(G_{r},L_{0})],$$
 we get
  $$    s\phi(G_r,Q_{s})=\sum_{i\in\mathbb{Z}}ia_{r,s,i}^{(2)}L_{i}+
 \sum_{j\in\mathbb{Z}}jb_{r,s,j}^{(2)}H_{j}.$$
 Thus,  $\phi(G_r,Q_{s})=0$.

 Assume that
$$ \phi(Q_r,Q_{s})=\sum_{i\in\mathbb{Z}}a_{r,s,i}^{(3)}L_{i}+
 \sum_{j\in\mathbb{Z}}b_{r,s,j}^{(3)}H_{j}$$
 for some $a_{r,s,i}^{(3)},b_{r,s,j}^{(3)}\in\mathbb{C}.$
 By
  $$\phi(Q_{r},[Q_{s},L_{0}])=[\phi(Q_{r},Q_{s}), L_{0}]+(-1)^{(|\phi|+|Q_{r}|)|Q_{s}|}[Q_{s},\phi(Q_{r},L_{0})],$$
 we get
  $$    s\phi(Q_r,Q_{s})=\sum_{i\in\mathbb{Z}}ia_{r,s,i}^{(3)}L_{i}+
 \sum_{j\in\mathbb{Z}}jb_{r,s,j}^{(3)}H_{j}.$$
 Thus,  $\phi(Q_r,Q_{s})=0$.

   Assume that
 $$\phi(H_m,Q_r)=\sum_{\alpha\in\mathbb{Z}+\frac{1}{2}}a_{m,r,\alpha}^{(4)}Q_{\alpha}+
 \sum_{\beta\in\mathbb{Z}+\frac{1}{2}}b_{m,r,\beta}^{(4)}G_{\beta}$$
 for some $a_{m,r,\alpha}^{(4)},b_{m,r,\beta}^{(4)}\in\mathbb{C}$.
By
$$\phi(H_{m},[Q_{r},G_{r}])=[\phi(H_{m},Q_{r}), G_{r}]+(-1)^{(|\phi|+|H_{m}|)|Q_{r}|}[Q_{r},\phi(H_{m},G_{r})],$$
 we get
 $$ \sum_{\alpha\in\mathbb{Z}+\frac{1}{2}}a_{m,r,\alpha}^{(4)}H_{\alpha+r}
 +\sum_{\beta\in\mathbb{Z}+\frac{1}{2}}b_{m,r,\beta}^{(4)}L_{\beta+r}=0.$$
  Thus,  $\phi(H_m,Q_r)=0$.

Assume that
$$\phi(Q_r,C_{1})=\sum_{\alpha\in\mathbb{Z}+\frac{1}{2}}a_{r,\alpha}G_{\alpha}
+\sum_{\beta\in\mathbb{Z}
+\frac{1}{2}}b_{r,\beta}Q_{\beta}.$$
From
$$\phi(Q_{r},[G_{s},C_{1}])=[\phi(Q_{r},G_{s}), C_{1}]+(-1)^{(|\phi|+|Q_{r}|)|G_{s}|}[G_{s},\phi(Q_{r},C_{1})],$$
 we get
 $$ \sum_{\alpha\in\mathbb{Z}+\frac{1}{2}}a_{r,\alpha}L_{\alpha+s}+\sum_{\beta\in\mathbb{Z}
 +\frac{1}{2}}b_{r,\beta}H_{\beta+s}=0.$$
  Thus, $\phi(Q_{r},C_{1})=0$.
 Similarly, we get
 $\phi(Q_{r},C_{2})=0$.
\end{proof}

\subsection{The N=1 super-$ {\rm BMS_3}$ algebra}

\begin{defi}
The {\rm N=1} super-${\rm BMS_3}$ algebra
$$\mathcal B=\bigoplus_{n\in\mathbb{Z}}\cc L_n\oplus\bigoplus_{n\in\mathbb{Z}}\cc W_n\oplus\bigoplus_{r\in\mathbb{Z}+\frac{1}{2}}\cc Q_r\oplus\cc C_1\oplus\cc C_2$$  is a Lie superalgebra with the following commutation relations:
\begin{align*}\label{brackets}
&[L_m, L_n]=(m-n)L_{m+n}+{1\over12}\delta_{m+n, 0}(m^3-m)C_1,\\
&[L_m, W_n]=(m-n)W_{m+n}+{1\over12}\delta_{m+n, 0}(m^3-m)C_2,\\
&[Q_r, Q_s]=2W_{r+s}+{1\over3}\delta_{r+s, 0}\left(r^2-\frac14\right)C_2,\\
&[L_m,  Q_r]= \left(\frac{m}{2}-r\right)Q_{m+r},\\
&[W_m, W_n]=[W_n,Q_r]=0, \quad [C_1,\mathcal B]=[C_2, \mathcal B]=0
\end{align*}
 for any $m, n\in\mathbb{Z}, r, s\in\mathbb{Z}+\frac12$.
\end{defi}
 The N=1 super-${\rm BMS_3}$ algebra $\mathcal B$ has the decomposition:
$\mathcal B=\mathcal B_{\bar0}\oplus\mathcal B_{\bar1},$
where
$$
\mathcal B_{\bar0}=\bigoplus_{n\in\mathbb{Z}}\cc L_n\oplus\bigoplus_{n\in\mathbb{Z}}\cc W_n\oplus\cc  C_1\oplus\cc  C_2,\quad \mathcal B_{\bar1}=\bigoplus_{r\in\mathbb{Z}+\frac12} \cc  Q_r.$$
\begin{theo}
 Every  symmetric super-biderivation of the N=1 super-${\rm BMS_3}$ algebra  $\mathcal B$ is trivial.
 \end{theo}
\begin{proof} Let $\phi$ be a super-biderivation of $\mathcal B$.
 From  Corollary  \ref{3.5}, we have
 $ \phi(L_{m},L_{n}) = \phi(L_{m},W_{n})=\phi(W_{m},W_{n})=0$, $\phi(L_{m},C_{1}) =\phi(L_{m},C_{2})=0$,
  $\phi(W_{m},C_{1}) =\phi(W_{m},C_{2})=0$ for all $ m,n\in \mathbb{Z}$.
  By Theorem \ref{main2}, we get $\varphi(L_m, Q_r)= \varphi(Q_r,L_{m})=0$. So we need to determine $\phi(Q_{r}, Q_{s}) $, $\phi(Q_{r},C_{1})$ and $\phi(Q_{r},C_{2})$.

Assume that
$$\phi(Q_{r},Q_{s})=\sum_{i\in \mathbb{Z}} a_{r,s,i}L_{i} +\sum_{j\in\mathbb{Z}} b_{r, s,j}W_{j}$$
for some  $a_{r, s,i}, b_{r, s,j}\in \mathbb{C}$.
 By
$$ \phi(Q_{r},[Q_{s},L_{0}])=[\phi(Q_{r},Q_{s}), L_{0}]+(-1)^{(|\phi|+|Q_{r}|)|Q_{s}|}[Q_{s},\phi(Q_{r},L_{0})],$$
we get
$$ s \sum_{i \in\mathbb{Z}}a_{r,s,i} L_{i}+ s\sum_{j \in \mathbb{Z}}b_{r,s,j}W_{j}=\sum_{i \in\mathbb{Z}}ia_{r,s,i} L_{i}+\sum_{j \in \mathbb{Z}}jb_{r,s,j} W_{j}. $$
 Thus,
 $ \phi(Q_{r},Q_{s})=0$.

Assume that
$$\phi(Q_r,C_{1})=\sum_{\alpha\in\mathbb{Z}+\frac{1}{2}}a_{r,\alpha}Q_{\alpha}$$
for some  $a_{r,\alpha}\in \mathbb{C}$.
By
$$\phi(Q_{r},[Q_{s},C_{1}])=[\phi(Q_{r},Q_{s}), C_{1}]+(-1)^{(|\phi|+|Q_{r}|)|Q_{s}|}[Q_{s},\phi(Q_{r},C_{1})],$$
 we get
 $ \sum_{\alpha\in\mathbb{Z}+\frac{1}{2}}a_{r,\alpha}W_{\alpha+s}=0$.
 Thus,  $\phi(Q_{r},C_{1})=0$.
 Similarly,  we get  $\phi(Q_{r},C_{2})=0$.
 It completes the proof.
\end{proof}

\subsection{The $N=2$ superconformal algebra}

In this subsection, we shall determine all symmetric biderivations of the $N=2$ superconformal algebra.

By definition, the $N=2$ Ramond algebra ${\frak N}$ is a Lie superalgebra over $\cc$ with a basis $\{L_m, H_m, G_m^\pm, C\mid m\in\mathbb{Z}\}$ and the following relations:
\begin{align*}
& [L_m,L_n]=(m-n)L_{n+m}+{1\over12}(m^3-m)C, \\
&[H_m,H_n]={1\over3}m\delta_{m+n,0}C,\quad  [L_m, H_n]=-nH_{m+n},\\
&[L_m,G_p^\pm]=(\frac{m}{2}-p)G_{p+m}^\pm,\  [H_m,G_p^\pm]=\pm G_{m+p}^\pm,\\
&[G_p^+,G_q^-]=2L_{p+q}+(p-q)H_{p+q}+{1\over3}(p^2-{1\over4})\delta_{p+q,0}C,\\
&[G_p^\pm,G_q^\pm]=[\frak N, C]=0, \ \forall m,n, p, q\in\mathbb{Z}.
\end{align*}

The  $N=2$ Ramond algebra ${\frak N}$  has the decomposition:
$\frak N=\frak N_{\bar0}\oplus\frak N_{\bar1},$
where$$
\frak N_{\bar0}=\bigoplus_{n\in\mathbb{Z}}\cc L_n\oplus\bigoplus_{n\in\mathbb{Z}}\cc H_n\oplus\cc  C,\quad \frak N_{\bar1}=\bigoplus_{r\in\mathbb{Z}+\frac12} \cc  G_r^{+}\oplus\bigoplus_{s\in\mathbb{Z}+\frac12}\cc G_s^{-}.$$
\begin{theo}
Every  symmetric super-biderivation of  the  $N=2$ Ramond algebra ${\frak N}$ is trivial.
\end{theo}
\begin{proof}  Let $\delta$ be a super-biderivation of $\frak N$.
From Corollary \ref{3.5}  and Theorem \ref{main2}, we get
$$\delta(L_m,L_n)=\sum_{k\in \mathbb{Z}}\mu_{k}H_{m+n+k},$$
$$\delta (L_m, H_n)=0,   \    \delta(H_m,H_n)=0,  \    \delta(G_p^\pm,G_q^\pm)=0,$$
$$  \delta(L_m,G_p^\pm)=0, \   \delta (L_m, C)=0,
  \    \delta(H_m,C)=0.$$
  So we need to determine
 $\delta(H_m,G_p^\pm) $, $\delta(G_p^+,G_q^-)$ and $\delta(G_p^\pm,C)$.

 Assume that
$$\delta(H_{m},G_p^+)=\sum_{i\in \mathbb{Z}} a_{m, p,i}^{(1)}G_i^++\sum_{j\in \mathbb{Z}} b_{m,p,j}^{(1)}G_j^-,$$
$$\delta(H_{m},G_p^-)=\sum_{i\in \mathbb{Z}} a_{m, p,i}^{(2)}G_i^++\sum_{j\in \mathbb{Z}} b_{m,p,j}^{(2)}G_j^-,$$
 where $a_{m,p,i}^{(1)}, b_{m,p,j}^{(1)},a_{m,p,i}^{(2)}, b_{m,p,j}^{(2)}\in \mathbb{C}$. By
$$ \delta(H_{m},[G_{p}^{+},H_0])=[\delta(H_{m},G_{p}^{+}), H_{0}]+(-1)^{(|\delta|+|H_{m}|)|G_{p}^{+}|}[G_{p}^{+},\delta(H_{m},H_{0})],$$
we get
$\sum_{j\in \mathbb{Z}} b_{m,p,j}^{(1)}G_{j+p}^-=0$.
 Thus,
 $ \delta(H_{m},G_p^+)=\sum_{i\in \mathbb{Z}} a_{m,p,i}^{(1)}G_i^+$.
 Similarly, we get
$ \delta(H_{m},G_p^-)=\sum_{j\in \mathbb{Z}} b_{m,p,i}^{(2)}G_i^-$.
  By
$$ \delta(H_{m},[G_{p}^{+}, G_{p}^{-}])=[\delta(H_{m},G_{p}^{+}),G_{p}^{-}]+
(-1)^{(|\delta|+|L_{m}|)|G_{p}^{+}|}[G_{p}^{+},\delta(H_{m},G_{p}^{-})],$$
we get  $\delta(H_{m},G_{p}^{\pm})=0$.

Assume that
$$\delta(G_p^+,G_q^-)=\sum_{i\in\mathbb{Z}}a_{p,q,i}^{(3)}L_{i}
+\sum_{j\in\mathbb{Z}}b_{p,q,j}^{(3)}H_{j},$$
 where $a_{p,q,i}^{(3)}, b_{p,q,j}^{(3)}\in \mathbb{C}$.
  By
$$ \delta(G_{p}^{+},[G_{p}^{-},H_{0}])=[\delta(G_{p}^{+},G_{q}^{-}), H_{0}]+(-1)^{(|\delta|+|G_{p}^{+}|)|G_{p}^{-}|}[G_{p}^{-},\delta(G_{p}^{+},H_{0})],$$
we get
$$\sum_{i\in\mathbb{Z}}a_{p,q,i}^{(3)}L_{i}
+\sum_{j\in\mathbb{Z}}b_{p,q,j}^{(3)}H_{j}=0.$$
Thus, $\delta(G_p^+,G_q^-)=0$.

By
$$ \delta(L_{m},[G_{p}^{+},L_{n}])=[\delta(L_{m},G_{p}^{+}), L_{n}]+(-1)^{(|\delta|+|L_{m}|)|G_{p}^{+}|}[G_{p}^{+},\delta(L_{m},L_{n})],$$
we get
 $$\sum_{k\in\mathbb{Z}}\mu_{k}G_{m+n+p+k}=0.$$
 Thus, $\mu_{k}=0$.

Assume that
$$\delta(G_p^{\pm},C)=\sum_{i\in \mathbb{Z}} a_{p,i}G_i^++\sum_{j\in \mathbb{Z}} b_{p, j}G_j^-,$$
 where $a_{p, i}, b_{p, j}\in \mathbb{C}$.
  By
$$ \delta(G_p^{\pm},[H_{m},C])=[\delta(G_{p}^{\pm},H_{m}),C]
+(-1)^{(|\delta|+|G_{p}^{\pm}|)|H_{m}|}[H_{m},\delta(G_{p}^{\pm},C)],$$
we get
$$\sum_{i\in\mathbb{Z}}a_{p,i}G_{m+i}^{+}-\sum_{j \in\mathbb{Z}}b_{p,j}G_{m+j}^{-}=0.$$
Thus, $\delta(G_p^{\pm},C)=0$.
It completes the proof.
\end{proof}

\subsection{The Heisenberg-Virasoro superalgebra}

In this subsection, we shall determine all symmetric biderivations of the Heisenberg-Virasoro superalgebra $S$.
 \begin{defi}
        The Heisenberg-Virasoro superalgebra $S$ is a Lie superalgebra whose even part $S_{\overline{0}}$ is spanned by $ \{ L_{m}, H_{m},C\mid m\in \mathbb{Z}\}$ and odd part $ S_{\overline{1}}$ is spanned by $\{G_{r} \mid  r\in \mathbb{Z}+\frac{1}{2}\}$, equipped with the following relations:
       \begin{align*}
        &[L_{m}, L_{n}]= (m-n)L_{m+n}+\frac{1}{12}\delta_{m+n, 0}(m^{3}-m)C,\\
       & [L_{m}, H_{n}]= -nH_{m+n},   \quad   [L_{m}, G_{r}]=-rG_{m+r},\\
      & [G_{r}, G_{s}]=2H_{r+s}, \, [S, C]=0, \forall  m, n \in \mathbb{Z}, r, s\in \mathbb{Z}+\frac{1}{2}.
       \end{align*}
\end{defi}

\begin{theo}\label{5.8}
Every symmetric super-biderivation of the Heisenberg-Virasoro superalgebra $S$ is as follows:
$$ \phi(L_{m},L_{n}) =\sum_{k\in \mathbb{Z}}\mu_{k}H_{m+n+k}, $$
for all  $m,n\in \mathbb{Z}$, where  $\mu_{k} \in \mathbb{C }$, the others are zero.
\end{theo}
\begin{proof} Let $\phi$ be a super-biderivation of $S$.
From [12, Theorem 3.2],  we get
$$\phi(L_m,L_n)=\sum_{k\in \mathbb{Z}}\mu_{k}H_{m+n+k},$$
$$\phi (L_m, H_n)=0,   \    \    \phi(H_{m},H_n)=0,$$
$$\phi (L_m, C)=\phi(H_{m},C)=0.$$
Next,  we  need to  determine $\phi(G_r,G_s)$, $\phi(L_m,G_r) $, $\phi(H_m,G_r) $ and $\phi(G_{r},C)$.
 By Theorem \ref{main2}, we can assume that
  $$\phi(G_r,G_s)=\sum_{i\in\mathbb{Z}}a_{r, s, i}^{(1)}L_{i}+\sum_{j\in \mathbb{Z}} b_{r,s,j}^{(1)}H_j,$$
   $$\phi(L_m,G_r)= \sum_{\alpha\in\mathbb{Z}+\frac{1}{2}}a_{m,r,\alpha}^{(2)}G_{\alpha}, \,
   \phi(H_m,G_r)= \sum_{\beta\in\mathbb{Z}+\frac{1}{2}}b_{m,r,\beta}^{(2)}G_{\beta},$$
   $$ \phi(G_r,C)=\sum_{\alpha\in \mathbb{Z}+\frac{1}{2}} a_{r,\alpha}G_\alpha$$
 for some $ a_{r,s,i}^{(1)},b_{r,s,j}^{(1)},a_{m,r,\alpha}^{(2)},b_{m,r,\beta}^{(2)},a_{r,\alpha}\in \mathbb{C}$.
By
$$ \phi(L_{m},[L_{0},G_{r}])=[\phi(L_{m},L_{0}), G_{r}]+(-1)^{(|\phi|+|L_{m}|)|L_{0}|}[L_{0},\phi(L_{m},G_{r})],$$
we get
$ r\sum_{\alpha \in\mathbb{Z}+\frac{1}{2}}a_{m,r,\alpha}^{(2)}G_{\alpha} =
 \sum_{\alpha\in\mathbb{Z}+\frac{1}{2}}a_{m,r,\alpha}^{(2)}\alpha G_{\alpha}$,
  then
   $a_{m,r,\alpha}^{(2)}=0$  for $ \alpha\neq r$.
    Thus,  we have
    \begin{equation}\label{bi-1}
    \phi(L_m,G_{r})=a_{m,r,r}^{(2)}G_{r}.
    \end{equation}
 By
$$ \phi(G_{r},[G_{s},L_{0}])=[\phi(G_{r},G_{s}),L_{0}]
+(-1)^{(|\phi|+|G_{r}|)|G_{s}|}[G_s,\phi(G_{r},L_{0})],$$
we get
 $$ s\sum_{i\in\mathbb{Z}}a_{r, s,i}^{(1)}L_{i}+s\sum_{j\in\mathbb{Z}}b_{r, s,j}^{(1)}H_{j}
 =\sum_{i\in\mathbb{Z}}ia_{r, s, i}^{(1)}L_{i}+\sum_{j\in\mathbb{Z}}jb_{r, s, j}^{(1)} H_{j}-a_{0, r,r}^{(2)}2H_{r+s}.$$
 Then we have $a_{r, s, i}^{(1)}=0$ and $b_{r, s, j}^{(1)}= 0$ for $j \neq r+s$. Thus, $\phi(G_r,G_{s}) =b_{r, s, r+s}^{(1)}H_{r+s}$. For
 $j=r+s$, we have $r b_{r, s, r+s}^{(1)}=2a_{0,r, r}^{(2)}$ for all $r, s\in\mathbb{Z}+\frac{1}{2}$.  By $\phi(G_r,G_{s})=-\phi(G_s,G_{r})$,  we have
  $a_{0,r, r}^{(2)}=0$ for all $r \in\mathbb{Z}+\frac{1}{2}$.  So $b_{r, s, r+s}^{(1)}=0$ for all $r, s\in\mathbb{Z}+\frac{1}{2}$.  Hence,
  $$ \phi(G_r,G_{s})=0, \forall r,s\in\mathbb{Z}+\frac{1}{2}.$$

 By $ \phi(G_{r},[G_{s},H_n])=[\phi(G_{r}, G_{s}),H_n]+(-1)^{(|\phi|+|G_{r}|)|G_{r}|}[G_s,\phi(G_{r},H_{n})]$,
we get
$ \sum_{\beta\in\mathbb{Z}+\frac{1}{2}}b_{n, r,\beta}^{(2)}=0$.
Then $$\phi(H_m,G_{r})=0, \forall m\in\mathbb{Z}, r\in\mathbb{Z}+\frac{1}{2}.$$

 By $ \phi(G_{r},[L_{m},G_{s}])=[\phi(G_{r},L_{m}), G_{s}]+(-1)^{(|\phi|+|G_{r}|)|L_{m}|}[L_{m},\phi(G_{r},G_{s})]$,
we get $a_{m, r, r}^{(2)}=0$. Combining with \eqref{bi-1}, we have
 $$\phi(L_m,G_{r})=0, \forall m\in\mathbb{Z}, r\in\mathbb{Z}+\frac{1}{2}.$$

By $\phi(G_{r},[G_{r},C])=[\phi(G_{r},G_{r}), C]+(-1)^{(|\phi|
  +|G_{r}|)|G_{r}|}[G_{r},\phi(G_{r},C)]$,
we get
$\sum_{\alpha\in \mathbb{Z}+\frac{1}{2}} 2a_{r,\alpha}H_{\alpha+r}=0$.
  Thus, we get
 $$\phi(G_{r},C)=0, \forall r\in\mathbb{Z}+\frac{1}{2}.$$
It completes the proof.
\end{proof}

\section{Post-Lie superalgebra structures on some Lie superalgebras}

 In this section, we shall study post-Lie superalgebra structures on  Lie superalgebras mentioned above.

\begin{defi}\label{66} Let $(L,[,])$ be a Lie superalgebra. A commutative post-Lie superalgebra structure on $L$ is a bilinear product $x\circ y$ on $L$ satisfying the following identities:
 \begin{align*}
& x\circ y =(-1)^{|x||y|}y\circ x , \\
 &   [x,y]\circ z=x\circ(y\circ z)-(-1)^{|x||y|}y\circ(x\circ z), \\
  &   x\circ [y,z]= [x\circ y,z]+(-1)^{|x||y|}[y,x\circ z], \ \forall x,y,z\in L.
   \end{align*}
    \end{defi}
 We also say $(L,[,],\circ)$ a commutative post-Lie superalgebra. A post-Lie superalgebra
  $(L, [,], \circ)$ is said to be trivial if $x\circ y=0$ for all $x, y \in L$.

 The following lemma shows the connection between commutative post-Lie superalgebras and symmertric biderivations of $L$.
  \begin{lemm}\label{6.2}
  Suppose that $(L,[,],\circ)$ is commutative post-Lie superalgebra. If we define a bilinear map $ f: L\times L\rightarrow L$ by $f(x,y)=x\circ y$ for all $x,y\in L$, then $f$ is a symmetric biderivation of $L$.
  \end{lemm}
  \begin{proof}
  For any $x, y, z \in L$,
   by Definition \ref{66}, we deduce that
 \begin{align*}
f([x,y],z)&=[x,y]\circ z=(-1)^{|z|(|x|+|y|)}z\circ [x,y]\\
 &=(-1)^{|z|(|x|+|y|)}([z\circ x, y]+(-1)^{|x||z|}[x,z\circ y])\\
 &=(-1)^{|z||y|}([x\circ z,y]+(-1)^{|z||y|}[x, z\circ y])\\
 &= (-1)^{|z||y|}[f(x, z),y]+[x, f(y, z)],
   \end{align*}
\begin{align*}
   f(x,[y, z])&=x\circ [y, z]=[x\circ y, z]+(-1)^{|x||y|}[y, x\circ z]\\
  &= [f(x, y),z]+(-1)^{|x||y|}[y, f(x, y)]. \\
           \end{align*}
           So $f$ is a symmetric  biderivation of $L$.

  \end{proof}
  \begin{theo}
  Any commutative post-Lie superalgebra structure on the super Virasoro algebra ${\rm SVir}$,  the super ${\rm W(2,2)}$ algebra ${\rm SW}$,
  the ${\rm N}=1$ super-${\rm BMS_3}$ algebra ${\mathcal B}$, the ${\rm N}=2$  superconformal algebra ${\frak N}$ and the Heisenberg-Virasoro superalgebra $S$  is trivial.
  \end{theo}
  \begin{proof}
    It is clear for the super Virasoro algebra ${\rm SVir}$,  the super ${\rm W(2,2)}$ algebra ${\rm SW}$,
  the ${\rm N}=1$ super-${\rm BMS_3}$ algebra ${\mathcal B}$, the ${\rm N}=2$  superconformal algebra ${\frak N}$,  since any symmetric biderivation on these algebra is trivial.

For the Heisenberg-Virasoro superalgebra $S$,  suppose that $( S,[,],\circ)$ is a commutative post-Lie superalgebra. By Lemma \ref{6.2} and Theorem \ref{5.8}, we know that
$x\circ y=0$ for all $x, y\in S$ except that
 $L_{m}\circ L_{n}=\sum_{k\in \mathbb{Z}}\mu_{k}H_{m+n+k}$ for any $m, n\in\mathbb{Z}$, where $\mu_k\in\mathbb C$.

By
$  [L_2,L_1]\circ L_3=L_2\circ (L_1\circ L_3)-L_1\circ (L_2\circ L_3)$, it is easy to
see that the left-hand side of the above equation contains an item
$ \mu_{k}H_{6+k}\neq 0$, whereas the right-hand side is equal to zero, which is a contradiction. Thus, we have $\mu_{k}=0$ for any $k\in \mathbb{Z}$.
That is,
 $x\circ y=0$ for all $x, y \in S$.
\end{proof}

\vskip30pt \noindent{\bf Acknowledgments}
 This work is partially supported by the NNSF (Nos. 12071405, 11971315, 11871249), and is partially supported by Xinjiang Uygur Autonomous Region graduate scientific research innovation project (No. XJ2021G021).

\end{document}